\documentclass{amsart}       

\usepackage{amssymb,latexsym} 

\usepackage{hyperref}                                     %
\makeatletter                                                      %
\def\@secnumfont{\bfseries}                           %
\makeatletter                                                      %

\theoremstyle{plain} 
\newtheorem{theorem}{\bf Theorem}
\newtheorem*{theorem*}{\bf Theorem}

\theoremstyle{definition}

\begin{document}

\address{Research and Exploratory Development Department}
\address{Johns Hopkins University Applied Physics Laboratory}

\numberwithin{theorem}{section}

\title{\bf Wallis's  Formula and the Probability Integral}

\maketitle                                 

\centerline{James R. Schatz}   

\bigskip
\section{\bf Introduction}

In 1655 John Wallis proved the following remarkable  assertion, 

\begin{equation}
\lim_{n \rightarrow \infty} \  \frac{2}{1} \cdot \frac{2}{3} \cdot \frac{4}{3}  \cdot \frac{4}{5}  \cdot  \frac{6}{5}  \cdot \frac{6}{7}
 \dots  \frac{2n}{2n - 1} \cdot \frac{2n}{2n + 1} = \frac{ \pi }{2}.   \vspace*{10pt} 
 \end{equation}
Wallis's formula is used in a number of methods of computing the improper integral

\begin{equation}
\int_0^\infty e^{-x^2} \ dx = \frac{ \sqrt{\pi} }{2}.    \vspace*{10pt} 
\end{equation}
This integral is often called the probability integral because of the important role it plays in the theory of the normal distribution. 
The goal of this expository note is to consider some of the elementary methods of proving these famous formulas and to explore the relationships between them.
As a first step in understanding the close connection between these two fascinating assertions, 
we will show that any proof of either (1) or (2) automatically proves both formulas. After establishing this important result, we will review some
elementary methods of proving Wallis's formula and the probability integral formula.

\bigskip

There are a number of simple variations of Wallis's formula that are equivalent to the basic form (1). 
Two of these alternative statements are useful  in the arguments that follow. As an easy exercise, the reader might like to 
prove that each of the following two assertions is equivalent to Wallis's formula as given in (1): 

\begin{equation}
\sqrt{\pi}  = \lim_{n \rightarrow \infty}  \  \frac{1}{\sqrt{n}} \cdot \Big( \frac{2}{1} \cdot \frac{4}{3} \cdot \frac{6}{5} \dots \frac{2n}{2n - 1} \Big)
\end{equation}
\begin{equation}
\frac{\sqrt{\pi}}{2} = \lim_{n \rightarrow \infty} \sqrt{n} \ \Big(  \frac{2}{3} \cdot \frac{4}{5} \cdot \frac{6}{7} \dots \frac{2n}{2n + 1} \Big). \vspace*{10pt} 
\end{equation}

\noindent
These and other variations of Wallis's formula are collected and proved in the final section of this note.

 \section{\bf Connecting Two Famous Formulas}
\vspace*{4pt}

\bigskip

The following theorem about a sequence of integrals closely connected to the moments of the normal distribution
is the key to understanding the close connection between Wallis's formula and the probability integral.

\bigskip

\begin{theorem}
For each integer $n \geq 0$ define
 
 $$E_n = \int_{0}^{\infty} x^n e^{-x^2} \ dx.  \vspace*{10pt} $$
 \noindent
 {\rm (a)} For all $n \geq 0$ the improper integral $E_n$ converges.

\vspace*{10pt}
 \noindent
{\rm (b)} For all $n \geq 0$, 
$$E_{n+2} = \frac{n+1}{2} E_n.$$

\vspace*{10pt}
\noindent 
{\rm (c)} For all $n \geq 1$, 
 $$E_{2n} = \frac{ 1 \cdot 3 \cdot 5 \cdot 7 \dots (2n-1)}{2^n} E_0. \vspace*{10pt} $$
 \noindent
{\rm (d)} For all $n \geq 0$,  
$$E_{2n+1} = \frac{n!}{2}. \vspace*{10pt} $$

 \noindent
{\rm (e)} For all $n \geq 1$, \  $E_n^2 \leq E_{n+1}E_{n-1}$. 
 
 \end{theorem}

 \begin{proof}
 
 First notice that for any real number $b > 0$

$$\int_0^b x \ e^{-x^2} \ dx \ = \  \frac{e^{-x^2}}{-2} \Big]_0^b \  = \  \frac{1}{2} - \frac{e^{-b^2}}{2}.    \vspace*{10pt} $$
Therefore, the improper integral $E_1$ converges to $1/2$.  
For all $x \geq 1$ we clearly have $ 0 \leq e^{-x^2}  \leq x  e^{-x^2}$.  So, by the Comparison Test for improper integrals, the 
probability integral $E_0$ also converges.
As an aside, notice that we cannot conclude that the value of the probability integral is bounded above by
$1/2$ because the inequality $e^{-x^2}  \leq x  e^{-x^2}$ is only true for $x \geq 1$.  

\bigskip

 Now, recall the general formula for integration by parts: 
 
 $$\int_a^b f(x) g'(x) \ dx = f(b)g(b) - f(a)g(a) - \int_a^b f'(x) g(x) \  dx.  \vspace*{10pt} $$
For $n \geq 0$, apply the  integration by parts  formula with 

$$f(x) = e^{-x^2}, \ \  f'(x) = (-2x) e^{-x^2}, \ \ g(x) = \frac{x^{n+1}}{n+1}, \ \ {\rm and} \ \ g'(x) = x^n  \vspace*{10pt} $$
to see that  for all real numbers $b > 0$,

$$ \int_0^b x^n e^{-x^2} \ dx  =  \frac{x^{n+1}}{n+1} e^{-x^2} \Big]_0^b \  - \  \int_0^b (-2x)  \frac{x^{n+1}}{n+1} e^{-x^2} \ dx. \vspace*{10pt} $$
Rearranging this equation gives us

$$\int_0^b  x^{n+2}  e^{-x^2} \ dx = \frac{n+1}{2}  \int_0^b x^n e^{-x^2} \ dx - \Big( \frac{1}{2} \Big) b^{n+1} e^{-b^2}. \vspace*{10pt} $$
Since we know that $E_0$ and $E_1$ converge, the formula above shows that $E_n$ converges for all $n \geq 0$ and that 
we have the reduction formula (b). 
The evaluations of $E_{2n}$ and $E_{2n+1}$ given in formulas (c) and (d) are easily obtained   from the reduction formula  by induction.

\bigskip

Now for the tricky step! For any real number $\lambda$ and all integers $n \geq 1$, 
\begin{align*}
E_{n+1} + 2 \lambda E_n + \lambda^2 E_{n-1} &= \int_0^\infty e^{-x^2} (x^{n+1} + 2 \lambda x^n + \lambda^2 x^{n-1} ) \ dx \\[10 pt]
                                                                                &=  \int_0^\infty e^{-x^2} x^{n-1} (x + \lambda)^2 \ dx \\[10 pt]
                                                                                &> 0.
\end{align*}

\vspace*{5pt}
\noindent
Therefore, the quadratic polynomial $E_{n+1} + 2 \lambda E_n + \lambda^2 E_{n-1}$ in $\lambda$ has no real roots and  its discriminant must be negative. That is, 

$$4\ E_n^2  - 4 \  E_{n+1} \  E_{n-1} < 0.  \vspace*{10pt} $$
Thus, for all integers $n \geq 1$ we have $E_n^2 <  E_{n+1} \ E_{n-1}$ and (e) is proved. This proof of (e) is due to Stieltjes, 1890.
\end{proof}

\bigskip

\begin{theorem} Any proof of  Wallis's formula or the probability integral formula proves both assertions. 
\end{theorem}

\begin{proof}
Define $E_n$ as in the previous theorem. For $n \geq 1$ we have 
$E_{2n}^2 < E_{2n +1} E_{2n -1}$. Using the formulas for $E_{2n-1}, E_{2n}$ and $E_{2n+1}$ we see that

$$\frac{ (1 \cdot 3 \cdot 5 \dots (2n -1) )^2}{2^{2n}} \ E_0^2 \  \leq \ \frac{ n! (n-1)!}{4}. $$
Thus, 
$$4 E_0^2 \  \leq \ \frac{ 2^{2n} \  n! \ n! } {n \ (1 \cdot 3 \cdot 5 \dots (2n-1))^2},  \vspace*{10pt} $$
or 
$$2 E_0 \ \leq \  \frac{1}{\sqrt{n}} \Big( \frac{2}{1} \cdot \frac{4}{3} \cdot \frac{6}{5} \dots \frac{2n}{2n-1} \Big). \vspace*{10pt} $$
For $n \geq 0$ we have  $E_{2n+1}^2 < E_{2n +2} E_{2n}$. Therefore,

$$\frac{ n! \ n!}{4} \ \leq \ \Big( \frac{ 1 \cdot 3 \cdot 5 \dots (2n+1)}{2^{n+1}} \Big) \ \Big( \frac{ 1 \cdot 3 \cdot 5 \dots (2n-1)}{2^n} \Big) \  E_0^2,  \vspace*{10pt} $$
or
$$ \frac{ 2^{n+1} \ 2^n \ n! \ n! }{ (1 \cdot 3 \cdot 5  \dots (2n-1))^2 (2n+1)} \leq \ 4 \ E_0^2. $$
Therefore, 

$$\sqrt{ \frac{2}{2n+1} }  \   \Big(    \frac{ n! \ 2^n}{ 1 \cdot 3 \cdot 5  \dots (2n-1)} \Big) \leq 2 \ E_0, \vspace*{10pt} $$
or
$$\sqrt{ \frac{2n}{2n+1} }  \ \frac{1}{\sqrt{n}} \ \Big( \frac{2}{1} \cdot \frac{4}{3} \cdot \frac{6}{5} \dots \frac{2n}{2n - 1} \Big) \leq 2 \ E_0. \vspace*{10pt} $$
This last inequality can also be rewritten as:

$$ \frac{1}{\sqrt{n}} \ \Big( \frac{2}{1} \cdot \frac{4}{3} \cdot \frac{6}{5} \dots \frac{2n}{2n - 1} \Big) \leq 2 \ E_0 \sqrt{ \frac{2n+1}{2n}}. \vspace*{10pt} $$
Thus, we have now shown that

$$ \sqrt{ \frac{2n}{2n+1} }  \ \frac{1}{\sqrt{n}} \ \Big( \frac{2}{1} \cdot \frac{4}{3} \cdot \frac{6}{5} \dots \frac{2n}{2n - 1} \Big) \  \leq 2 \ E_0 \
\leq  \frac{1}{\sqrt{n}} \Big( \frac{2}{1} \cdot \frac{4}{3} \cdot \frac{6}{5} \dots \frac{2n}{2n-1} \Big)  \vspace*{10pt} $$
and that 
$$2 E_0 \ \leq \frac{1}{\sqrt{n}} \ \Big( \frac{2}{1} \cdot \frac{4}{3} \cdot \frac{6}{5} \dots \frac{2n}{2n - 1} \Big) \leq 2 \ E_0 \ \sqrt{ \frac{2n+1}{2n}}. \vspace*{10pt} $$
By two applications of the squeeze theorem for limits we see that assertion (3)  is true if and only if the probability integral $E_0$ has the value $\sqrt{\pi} / 2$.
\end{proof}

\section{\bf Wallis Integrals}

 The integrals discussed in this section are remarkable for their simplicity and their power.
The properties of these integrals and the relationships between them are the basis of proofs for
Wallis's formula and the probability integral formula.

\bigskip

 \begin{theorem}
 For each integer $n \geq 0$ define
 
 $$I_n = \int_{0}^{\pi / 2} \cos^n x \ dx.$$
 \noindent
 {\rm (a)} For all $n \geq 2$, 
 $$I_n = \frac{n-1}{n} I_{n-2}.$$
 
 \noindent
{\rm (b)} For all $n \geq 0$,  \
$$n \ I_n I_{n-1} = \frac{\pi}{2}.$$

 \noindent
{\rm  (c) } For all $n \geq 1$, 
 $$I_{2n} = \frac{ \pi }{2} \cdot  \frac{1}{2} \cdot \frac{3}{4} \cdot \frac{5}{6}  \dots  \frac{2n-1}{2n}. \vspace*{10pt} $$
 \noindent
{\rm (d)} For all $n \geq 1$, 
 $$I_{2n+1} = 1 \cdot \frac{2}{3} \cdot \frac{4}{5} \cdot \frac{6}{7} \cdot \frac{8}{9}  \dots  \frac{2n}{2n+1}. \vspace*{10pt} $$
 \noindent
{\rm (e)} For all $n \geq 1$,  
$$I_{n+1} \leq I_n.$$

  \noindent
{\rm (f)} 
$$\lim_{n \rightarrow \infty} \sqrt{n} \  I_n = \sqrt{ \frac{\pi}{2} }. \vspace*{10pt} $$
 \end{theorem}
 
 \begin{proof}
 Assume that $n \geq 2$ and apply  the integration by parts formula with 
 
$$f(x) = \cos^{n-1} x,   \ \ \   f'(x) = - (n-1) (\cos^{n-2} x)(\sin x), \vspace*{5pt} $$
$$g(x) = \sin x, \ \ {\rm and} \ \  g'(x) = \cos x$$
 to obtain
\begin{align*} 
\int_{0}^{\pi / 2} \cos^n x \ dx  &= (n-1) \int_{0}^{\pi / 2} (\cos^{n-2} x ) (\sin^2 x) \ dx \\
                                                         &= (n-1) \int_{0}^{\pi / 2} (\cos^{n-2} x ) (1- \cos^2 x) \ dx \\
                                                         &= (n-1) \int_{0}^{\pi / 2}  \cos^{n-2} x \ dx - (n-1) \int_{0}^{\pi / 2} \cos^n x \ dx.
\end{align*}  

 \vspace*{5pt}
\noindent
Thus, $n I_n = (n-1) I_{n-2}$ and the reduction formula (a) now follows.
Since $I_0 = \pi / 2$ and $I_1 = 1$, property (b) and the evaluations of $I_{2n}$ and $I_{2n + 1}$ given in (c) and (d) follow immediately
from the reduction formula by induction. 

\bigskip

Notice that for all $x$ such that 
$0 \leq x \leq \pi / 2$ we have $0 \leq \cos x \leq 1$. Therefore,  for all $x$ such that 
$0 \leq x \leq \pi / 2$ and  all 
$n \geq 1$, we have $0 \leq \cos^{n+1} x \leq \cos^n  x $.  The inequality 
$I_{n+1} \leq I_n$ given in (e)  now follows.

\bigskip

For $n \geq 1$,  we have $I_n \leq I_{n-1}$. Multiply through by $n I_n$ to obtain
$$n \ I_n^2 \ \leq \  n \ I_n\ I_{n-1} =  \frac{\pi}{2} .$$
For $n \geq 0$ we have $I_{n+1} \leq I_n$. Multiply through by $n I_n$ to obtain

$$n \ I_{n+1} \ I_n \leq n \  I_n^2.$$
Now, using property (b) we have
$$ n \ I_{n+1} \ I_n = \Big( \frac{n}{n+1}\Big)  \  (n+1) \ I_{n+1} \ I_n = \Big( \frac{n}{n+1} \Big) \  \frac{\pi}{2}.$$
Therefore,
$$\Big( \frac{n}{n+1} \Big) \  \frac{\pi}{2} \  \leq \ n \ I_n^2 \  \leq \frac{\pi}{2} . \vspace*{10pt} $$
Taking square roots and applying the squeezing theorem for limits proves (f).
\end{proof}

\bigskip

The integrals in the preceding theorem are called Wallis integrals since they lead to the following  easy 
proof of Wallis's formula.

\bigskip

\begin{theorem} (Wallis's Formula)
$$ \lim_{n \rightarrow \infty} \  \frac{2}{1} \cdot \frac{2}{3} \cdot \frac{4}{3}  \cdot \frac{4}{5}  \cdot  \frac{6}{5}  \cdot \frac{6}{7}
 \dots  \frac{2n}{2n - 1} \cdot \frac{2n}{2n + 1} = \frac{ \pi }{2}.   \vspace*{10pt} $$
\end{theorem}

\begin{proof} Let $I_n$ be the integral defined in the previous theorem. Then, for all $n \geq 1$, the inequality  $I_{2n + 1} \leq I_{2n}$ and the 
formulas for $I_{2n}$ and  $I_{2n+1}$ imply that

$$\Big( \frac{2}{3}  \cdot \frac{4}{5} \cdot  \frac{6}{7} \cdot  \frac{8}{9} \dots \frac{2n}{2n + 1} \Big) \ \leq \
 \frac{\pi}{2} \ \Big(\frac{1}{2}  \cdot \frac{3}{4} \cdot  \frac{5}{6} \cdot   \frac{7}{8}   \dots   \frac{2n - 1}{2n}\Big).  \vspace*{10pt} $$
 Equivalently, 
$$\Big( \frac{2}{1} \cdot \frac{2}{3} \cdot \frac{4}{3}  \cdot \frac{4}{5}  \cdot  \frac{6}{5}  \cdot \frac{6}{7}
 \dots  \frac{2n}{2n - 1} \cdot \frac{2n}{2n + 1} \Big) \leq \frac{ \pi }{2}.   \vspace*{10pt} $$
For all integers $n \geq 1$ we also have $I_{2n} \leq I_{2n-1}$ and so

$$ \frac{\pi}{2} \ \Big( \frac{1}{2}  \cdot \frac{3}{4} \cdot  \frac{5}{6} \cdot   \frac{7}{8}   \dots   \frac{2n - 1}{2n} \Big)    \leq
\Big( \frac{2}{3}  \cdot \frac{4}{5} \cdot  \frac{6}{7} \cdot  \frac{8}{9} \dots \frac{2n-2}{2n - 1} \Big).  \vspace*{10pt} $$
Equivalently, 
$$ \Big( \frac{2n}{2n+1} \Big) \frac{\pi}{2} \leq \Big( \frac{2}{1} \cdot \frac{2}{3} \cdot \frac{4}{3}  \cdot \frac{4}{5}  \cdot  \frac{6}{5}  \cdot \frac{6}{7}
 \dots  \frac{2n}{2n - 1} \cdot \frac{2n}{2n + 1} \Big). $$ 
Therefore, 

$$\Big( \frac{2n}{2n+1}  \Big)  \frac{\pi}{2} \leq \Big( \frac{2}{1} \cdot \frac{2}{3} \cdot \frac{4}{3}  \cdot \frac{4}{5}  \cdot  \frac{6}{5}  \cdot \frac{6}{7}
 \dots  \frac{2n}{2n - 1} \cdot \frac{2n}{2n + 1} \Big) \leq \frac{\pi}{2}. \vspace*{10pt} $$
By the squeezing theorem for limits the proof is complete. 
\end{proof}

\bigskip

Now that we have a proof of Wallis's  formula,  the evaluation of the probability integral (2) follows from Theorem 2.2. The proof of Wallis's
formula given above is  a well-known. The Wallis integrals can also be used to prove Wallis's formula as expressed in version (4).

\bigskip

\begin{theorem} (Wallis's Formula)
$$\frac{\sqrt{\pi}}{2} = \lim_{n \rightarrow \infty} \sqrt{n} \ \Big(  \frac{2}{3} \cdot \frac{4}{5} \cdot \frac{6}{7} \dots \frac{2n}{2n + 1} \Big). \vspace*{5pt} $$
\end{theorem}

\begin{proof}
From property (f) of Theorem 3.1 we have
$$\lim_{n \rightarrow \infty} \sqrt{2n+1} \  I_{2n+1} = \sqrt{ \frac{\pi}{2} }. $$
The theorem now follows easily from the observation that

$$ \sqrt{n} \ \Big(  \frac{2}{3} \cdot \frac{4}{5} \cdot \frac{6}{7} \cdot \frac{8}{9}  \dots  \frac{2n}{2n+1} \Big) = \frac{1}{\sqrt{2}} \  \sqrt{\frac{2n}{2n+1}} \ \sqrt{2n+1} \  I_{2n+1}. $$
\end{proof}

\bigskip

Wallis integrals sometimes  appear in a disguised form. 
The next theorem reveals two of the disguises.

 \bigskip
 
 \begin{theorem}
  {\rm (a)} For all integers $n \geq 1$, 

$$\int_0^{\infty}  \frac{1}{ (1 + x^2)^{n} }  \ dx = I_{2n-2}.  $$

\noindent
 {\rm (b)} For all integers $n \geq 1$, 

$$\int_0^1 (1 - x^2)^n \ dx = I_{2n +1}. \vspace*{5pt}$$
 \end{theorem}
 
 \begin{proof}
First, recall the general formula for substitution. 
Let $\phi \colon [a,b] \rightarrow [c,d]$ be a continuously differentiable function with $\phi (a) = c$ and $\phi (b) = d$ and let $f$ be a 
continuous function on $[c,d]$. Then

$$\int_a^b f( \phi(x) ) \phi'(x) \ dx = \int_c^d f(u) \ du. \vspace*{10pt} $$
Let $r$ be any real number such that $0 < r < \pi / 2$. Define $\phi \colon [0, r] \rightarrow [0, \tan r]$, $\phi (x) = \tan x$. Then $\phi'(x) = \sec^2 x$. Let
$f(u) = 1 / (1 + u^2)^n$. Then

$$\int_0^r \frac{1}{ (1 + \tan^2 x)^{n} } \sec^2 x  \ dx = \int_0^{\tan r}   \frac{1}{ (1 + u^2)^{n} } \ du. \vspace*{10pt} $$
Now recall that $1 + \tan^2 x = \sec^2 x$ and, of course, $\sec x  = 1 / \cos x$. 
Therefore, the equation above can be rewritten as
$$\int_0^{\tan r}   \frac{1}{ (1 + x^2)^{n} } \ dx = \int_0^r \cos^{2n -2} x \ dx.  \vspace*{10pt} $$
Taking the limit as $r$ approaches $ \pi / 2$ from the left proves that  the improper integral in statement (a) of the theorem exists and equals $I_{2n-2}$. 
 
 \bigskip
 
To prove (b), define $\phi \colon [0, \pi / 2] \rightarrow [0, 1]$ by $\phi (x) = \sin x$. Then $\phi'(x) = \cos x$.  For $f(u) = (1 - u^2)^n$ the 
substitution formula yields

$$\int_0^{\frac{\pi}{2}} (1 - \sin^2 x)^n \cos x \ dx = \int_0^1 (1 - u^2)^n \ du. \vspace*{5pt}$$
This equation can be rewritten as 

$$\int_0^1 (1 - x^2)^n \ dx = \int_0^{\frac{\pi}{2}} \cos^{2n + 1} x \ dx = I_{2n+1}. $$
\end{proof}

As an important special case of the  preceding theorem we have:

$$\int_0^\infty \frac{1}{1 + x^2} \ dx = \frac{ \pi }{2}.  \vspace*{5pt} $$
Moreover, Theorem 3.1 and Theorem 3.4  show that 
for each integer $n \geq 2$, 

$$ \int_0^\infty \frac{1}{ (1 + x^2)^{n} } \ dx = \frac{\pi}{2} \cdot \frac{1}{2} \cdot \frac{3}{4} \dots \frac{2n-3}{2n-2}, \vspace*{10pt} $$
and, for each integer $n \geq 1$, 

$$\int_0^1 (1 - x^2)^n \ dx = \frac{2}{3} \cdot \frac{4}{5} \dots  \frac{2n}{2n + 1}. \vspace*{5pt}$$

\bigskip

There is a proof of the probability integral formula that appears as a series of exercises in the book {\it Calculus} by
Michael Spivak. A variation of this approach is given in the next two theorems.

 \bigskip
 
 \begin{theorem}
 For all integers $n \geq 1$

$$ \int_0^1 (1 - x^2)^n \ dx \  \leq \  \frac{1}{\sqrt{n}} \int_0^{\sqrt{n}} e^{- x^2} \ dx \ \leq \  \int_0^\infty \frac{1}{(1 + x^2)^n} \ dx.  \vspace*{10pt}$$

 \end{theorem}
 
 \begin{proof}
For all $x \geq 0$,   $1 + x \leq e^x$. Hence,  $1 + x^2 \leq e^{x^2}$ and 
$$e^{- x^2} \leq \frac{1}{1 + x^2}. $$
Therefore, for any  integer $n \geq 1$, 

$$\int_0^1 e^{- n x^2} \ dx \leq \int_0^1 \frac{1}{(1 + x^2)^n} \ dx  \leq  \int_0^\infty \frac{1}{(1 + x^2)^n} \ dx.  \vspace*{10pt}$$
Applying the substitution $u = x \sqrt{n}$  we have

$$\int_0^1 e^{-n x^2} \ dx =   \frac{1}{\sqrt{n}} \int_0^{\sqrt{n}} e^{- u^2} \ du.  \vspace*{10pt}$$
This proves the right hand inequality of the theorem.

\bigskip

For $x \geq 0$ we know that 
$e^{-x} \leq e^0 = 1$, so we have the inequality

$$\int_0^x e^{-t} \ dt \  \leq \  \int_0^x 1 \ dt. \vspace*{10pt} $$
That is,  $1 - e^{-x}  \ \leq x$. Therefore, for all $x \geq 0$,  we have  $1 - x \   \leq \ e^{-x} $. Thus, 
for any $x$ such that $0 \leq x \leq 1$ we have 
$0 \leq 1 - x^2 \   \leq \ e^{-x^2} $. Now, for any integer $n \geq 1$, we have
$$\int_0^1 (1 - x^2)^n \ dx \leq \int_0^1 e^{-nx^2} \ dx =  \frac{1}{\sqrt{n}} \int_0^{\sqrt{n}} e^{- u^2} \ du.  \vspace*{10pt}$$
This proves the left hand inequality of the theorem.
 \end{proof}
 
 \bigskip
 
 \begin{theorem} (Probability Integral)
 $$ \int_0^\infty e^{-x^2} \ dx = \frac{ \sqrt{\pi} }{2}.    \vspace*{10pt} $$
 \end{theorem}
 
 \begin{proof}
For $n \geq 1$ the inequalities of the previous theorem can be expressed in terms of Wallis integrals as: 
$$\sqrt{n} \ I_{2n+1} \leq \int_0^{\sqrt{n}} e^{- x^2} \ dx \ \leq \sqrt{n} \ I_{2n-2}. $$
Therefore, 
$$\sqrt{\frac{2n}{2n+1}} \  \sqrt{2n+1} \  \ I_{2n+1} \leq \sqrt{2} \int_0^{\sqrt{n}} e^{- x^2} \ dx \ \leq \sqrt{\frac{2n}{2n-2}} \  \sqrt{2n-2} \  \ I_{2n-2}. \vspace*{5pt}$$
The theorem now follows from property (f) in Theorem 3.1. 
 \end{proof}
 
\bigskip

\section{\bf Another Approach to the Probability Integral}

There is a very nice approach to computing the probability integral that  requires only 
the most basic theorem on differentiation under the integral sign. This version is due to Borwein and Borwein
and appears as an exercise in the book {\it Irresistible Integrals} by Boros and Moll. First we recall the theorem
on differentiation under the integral sign.

\bigskip

Let
$f(t,x)$ be a real-valued  function defined and continuous on $[a,b] \times [c,d]$ where $a < b$ and $c < d$. 
Assume that the partial derivative  $D_1 f(t,x)$ exists and is continuous on $[a,b] \times [c,d]$.  
For each $t$ in $[a,b]$ let 
$$F(t) = \int_c^d f(t,x) \ dx.  $$
Then $F(t)$ is differentiable on $[a,b]$, with one-sided derivatives at $a$ and $b$, and  for each $t$ in $[a,b]$, 
$$F'(t) = \int_c^d D_1f(t,x) \ dx. \vspace*{5pt} $$
For our purposes define
$$F(t) = \int_0^1 \frac{e^{-t^2 (1 + x^2)}}{1 + x^2} \ dx.$$
It is easy to see that this integral satisfies the conditions of the theorem for all 
$(t,x)$ in any rectangle $[0,r] \times [0,1]$, where $r$ is any positive real number. Therefore, 

$$F'(t) = \int_0^1 (-2t) e^{-t^2 (1 + x^2)} \ dx \ =  \ -2 e^{- t^2} \int_0^1t \ e^{- (tx)^2} \ dx.  \vspace*{10pt} $$
Now carry out the substitution $\phi \colon [0,1] \rightarrow [0,t]$, \  $\phi (x) = tx$, \ $\phi'(x) = t$, with 
$f(u) = e^{-u^2}$ to obtain

$$F'(t) =  -2 e^{- t^2} \int_0^t  e^{- u^2} \ du = -2 e^{- t^2} \int_0^t  e^{- x^2} \ dx.   \vspace*{10pt} $$
Since the real number $r$ is fixed but arbitrary this formula for $F'(t)$ is valid for all $t \geq 0$. 
There is one additional fact about $F(t)$ that we need. For any real numbers $t$ and  $x$, $e^{-t^2(1 + x^2)} \leq e^{-t^2}$. Therefore

$$0 \leq \ F(t) \leq  \ e^{-t^2}  \int_0^1 \frac{1}{1+x^2} \ dx =  e^{-t^2} \  \frac{\pi}{4}.  \vspace*{10pt} $$
Thus, $0 \leq F(t) \leq \pi / 4$  for all $t \geq 0$ and

$$\lim_{t \rightarrow \infty} F(t) = 0.$$

Now define

$$G(t) = \Big( \int_0^t  e^{- x^2} \ dx \Big)^2. \vspace*{10pt} $$
By the Fundamental Theorem of Calculus, 

$$G'(t) = 2 e^{- t^2} \int_0^t  e^{- x^2} \ dx.  \vspace*{10pt} $$
Therefore, for all $t \geq 0$, $F'(t) + G'(t) = 0$ and so $F(t) + G(t)$ is a constant.
Since 
$$F(0) = \int_0^1 \frac{1}{1+x^2} \ dx =  \frac{\pi}{4}  \vspace*{10pt} $$
and $G(0) = 0$ it must be the case that $F(t) + G(t) = \pi / 4$ for all $t \geq 0$. 
Thus, for all $t \geq 0$, 

$$ \int_0^t  e^{- x^2} \ dx = \sqrt{ \frac{\pi}{4} - F(t) }.   \vspace*{10pt} $$
Taking the limit as $t \rightarrow \infty$ we see that the probability integral converges to 
$\sqrt{ \pi} / 2$. 
Notice that by Theorem 2.2 this proof of the probability integral formula  provides another proof of Wallis's formula.

\bigskip

\section{\bf Variations on Wallis's Formula}

\bigskip

There are a number of variations on Wallis' Formula that are very useful in applications. Let
 $\{ a_n \}$ denote the sequence in Wallis' Formula, 

\begin{equation}
a_n = \frac{2}{1} \cdot \frac{2}{3} \cdot \frac{4}{3} \cdot \frac{4}{5} \cdot \frac{6}{5} \cdot \frac{6}{7} \dots  
\frac{2n}{2n - 1} \cdot \frac{2n}{2n +1}. \vspace*{10pt} 
\end{equation}
The terms of this sequence can be rewritten in two ways: 

\begin{equation}
a_n = \Big( \frac{2}{1} \cdot \frac{4}{3}  \cdot \frac{6}{5}  \dots \frac{2n}{2n - 1} \Big)^2  \cdot \frac{1}{2n + 1} 
\end{equation}
and
\begin{equation}
  \ \ a_n = \Big( \frac{2}{3} \cdot \frac{4}{5}  \cdot \frac{6}{7}  \dots \frac{2n}{2n + 1} \Big)^2  \cdot (2n + 1).  
\end{equation}  
\vspace*{5pt} 

\noindent
These alternative formulas for $a_n$ will be useful in the variations that follow.

\bigskip

\noindent
{\bf Variation 1:} Wallis's Formula is equivalent to the assertion: 

$$\sqrt{\pi}  = \lim_{n \rightarrow \infty}  \  \frac{1}{\sqrt{n}} \cdot \Big( \frac{2}{1} \cdot \frac{4}{3} \cdot \frac{6}{5} \dots \frac{2n}{2n - 1} \Big).  \vspace*{10pt} $$

\begin{proof}
Starting with the expression (6) and taking square roots we have

$$\sqrt{a_n} =   \Big( \frac{2}{1} \cdot \frac{4}{3}  \cdot \frac{6}{5}  \dots \frac{2n}{2n - 1} \Big)  \cdot \frac{1}{ \sqrt{2n + 1}}. $$
Therefore,
 $$\ \frac{1}{\sqrt{n}} \cdot    \Big( \frac{2}{1} \cdot \frac{4}{3}  \cdot \frac{6}{5}  \dots \frac{2n}{2n - 1} \Big)  = \sqrt{a_n}  \cdot \sqrt{ \frac{2n+1}{2n} } \cdot \sqrt{2}.  \vspace*{10pt} $$
Hence,  Wallis's formula as expressed in  (1) implies Variation 1. The proof of the converse is similar.
\end{proof}

\bigskip
\noindent
{\bf Variation 2:} Wallis's Formula is equivalent to the assertion: 

$$ \frac{\sqrt{\pi}}{2} = \lim_{n \rightarrow \infty} \sqrt{n} \ \Big(  \frac{2}{3} \cdot \frac{4}{5} \cdot \frac{6}{7} \dots \frac{2n}{2n + 1} \Big). \vspace*{10pt} $$

\begin{proof}
Starting with the expression (7) and taking square roots we have

$$\sqrt{a_n} =   \Big( \frac{2}{3} \cdot \frac{4}{5}  \cdot \frac{6}{7}  \dots \frac{2n}{2n + 1} \Big)  \cdot \sqrt{2n + 1}. $$
Therefore,
$$ \sqrt{n} \ \Big(  \frac{2}{3} \cdot \frac{4}{5} \cdot \frac{6}{7} \dots \frac{2n}{2n + 1} \Big) = \sqrt{a_n}  \cdot  \sqrt{\frac{2n}{2n+1}} \cdot \frac{1}{\sqrt{2}}.  \vspace*{10pt} $$
Hence, Wallis's formula as expressed in  (1) implies Variation 2. The proof of the converse is similar.
\end{proof}

\bigskip
\noindent
{\bf Variation 3:} Wallis's Formula is equivalent to the assertion: 

$$\frac{1}{ \sqrt{\pi}} =  \lim_{n \rightarrow \infty}  \ \sqrt{ n } \cdot \Big( \frac{1}{2} \cdot \frac{3}{4} \cdot \frac{5}{6} \dots \frac{2n-1}{2n} \Big).  \vspace*{10pt} $$

\begin{proof}
This result is an immediate consequence of Variation 1.
\end{proof}

\bigskip
\noindent
{\bf Variation 4:} Wallis's Formula is equivalent to the assertion: 

$$\lim_{n \rightarrow \infty} \Big[ \ \Big( 1 - \frac{1}{2^2} \Big) \Big( 1 - \frac{1}{4^2} \Big) \Big( 1 - \frac{1}{6^2} \Big) \dots \Big( 1 - \frac{1}{ {(2n)}^2} \Big) \ \Big]  = 
\frac{2}{\pi}.   $$

\begin{proof}
The term $a_n$ in formula (5) can be written as

$$a_n =   \prod_{k=1}^n \  \frac{2k}{2k - 1} \cdot \frac{2k}{2k + 1}.    \vspace*{10pt} $$
Now, simply invert this formula and notice  that 

$$\frac{2k -1}{2k} \cdot \frac{2k +1}{2k} = \  \frac{(2k)^2 - 1}{(2k)^2} \  = \ 1 - \frac{1}{(2k)^2}. \vspace*{10pt} $$
The equivalence of Wallis's formula as given in (1) and Variation 4 is now clear.
\end{proof}

\bigskip
\noindent
{\bf Variation 5:} Wallis's Formula is equivalent to the assertion: 

$$ \lim_{n \rightarrow \infty} \  \frac{ (2n)!  \ \sqrt{\pi n} } { \  (n!)^2 \  2^{2n} }= 1.  \vspace*{10pt} $$

\begin{proof}
This limit is simply another way of writing Variation 3. 
\end{proof}

\bigskip

Variation 5 can be rewritten as 
an asymptotic formula for the central probability in the binomial distribution for an even number of trials,

$$\binom{2n}{n} \cdot \frac{1}{2^{2n}} \  \sim \  \frac{1}{\sqrt{\pi n }},   \vspace*{10pt} $$
or as an asymptotic formula for the central binomial coefficient,
\bigskip                 

$$\binom{2n}{n}  \  \sim \  \frac{2^{2n}}{\sqrt{\pi n }}.  \vspace*{10pt} $$

\bigskip
 

\bigskip

September 2019

\bigskip
email: schatzjames.math@gmail.com
                                               
\end{document}